\documentclass[11pt,a4paper]{amsart}
\usepackage[utf8]{inputenc}
\usepackage{amsmath}
\usepackage{amsfonts}
\usepackage{amssymb}
\usepackage{amsthm}
\usepackage{cases}
\usepackage[margin=.8in]{geometry}
\usepackage{tkz-euclide}

\newtheorem{theorem}{Theorem}
\newtheorem{corollary}[theorem]{Corollary}

\newtheorem{lemma}[theorem]{Lemma}

\newtheorem{rem}[theorem]{Remark}

\usepackage{hyperref}
\hypersetup{
  colorlinks = true, 
  urlcolor = blue, 
  linkcolor = blue, 
  citecolor = blue 
}

\DeclareMathOperator{\R}{\mathbb{R}}

\newcommand{\dd}{\ \mathrm{d}}
\newcommand{\del}{\partial}

\renewcommand{\rho}{\varrho}
\renewcommand{\phi}{\varphi}

\allowdisplaybreaks
\usepackage{esint}

\title{On two Kuznetsov's conjectures}
\author{Florian Oschmann}
\address{Institute of Mathematics, Czech Academy of Sciences,
\v Zitn\'a 25, 115 67 Praha 1, Czech Republic.}
 \email{oschmann@math.cas.cz}
 \date{\today}

\begin{document}
\maketitle
\begin{abstract}
We provide a proof and a counterexample to two conjectures made by N.~Kuznetsov.
\end{abstract}
\section{Introduction}
Mean value properties (MVP) as well as weighted MVP play a crucial role in the theory of partial differential equations. In the recent paper \cite{Kuznetsov2022}, the author investigates MVP for so-called $(\mu-)$panharmonic functions satisfying
\begin{align}\label{pan}
\Delta u - \mu^2 u=0 \text{ in } \Omega, \quad \mu\in \R\setminus{\{0\}},
\end{align}
where $\Omega\subset \R^2$ is a domain, and $\Delta=\del_{x_1}^2 + \del_{x_2}^2$ is the two-dimensional Laplacian. One of his main results reads as follows:
\begin{theorem}[{\cite[Theorem 3]{Kuznetsov2022}}]\label{theo1}
Let $\Omega\subset\R^2$ be a domain. If $u$ is panharmonic in $\Omega$, then
\begin{align}\label{i1}
a(\mu r) u(x) = \fint_{D_r(x)} u(y) \log\frac{r}{|x-y|}\dd y, \quad a(t)=\frac{2[I_0(t)-1]}{t^2},
\end{align}
for any $r>0$ such that $\overline{D_r(x)}=\{y\in \R^2: |x-y|\le r\}\subset \Omega$. Here, we denoted
\begin{align*}
\fint_{D_r(x)} u(y) \dd y = \frac{1}{\pi r^2}\int_{D_r(x)} u(y) \dd y,
\end{align*}
and $I_0(t)$ is the modified Bessel function of the first kind of order zero.
\end{theorem}

As $\mu\to 0$, one should expect that solutions to \eqref{pan} formally converge to a harmonic function. Moreover, as $a(0)=\lim_{t\to 0} a(t)=\frac12$, one shall also think that \eqref{i1} turns into
\begin{align*}
\Delta u=0 \Rightarrow \frac12 u(x)=\fint_{D_r(x)} u(y) \log\frac{r}{|x-y|} \dd y.
\end{align*}
Indeed, this is conjectured in \cite[Remark 1]{Kuznetsov2022}. Unfortunately, this claim is not proven there. The objective of the present note is to provide a short proof of this fact in any dimension $d \geq 2$. To make things precise, we will show:
\begin{lemma}\label{lem}
Let $d\geq 2$, $\Omega\subset \R^d$ be a domain, $\omega_d=\pi^\frac{d}{2}/\Gamma(\frac{d}{2}+1)$ be the volume of the $d-$dimensional unit ball, and let $\Delta u=0$ in $\Omega$. Then, for any $x\in \Omega$ and any $r>0$ such that $\overline{D_r(x)} \subset \Omega$, it holds
\begin{align*}
\frac1d u(x)=\fint_{D_r(x)} u(y) \log\frac{r}{|x-y|} \dd y = \frac{1}{\omega_d r^d} \int_{D_r(x)} u(y) \log\frac{r}{|x-y|} \dd y.
\end{align*}
\end{lemma}

Going even further, by the properties of the Bessel function $I_0(t)$, the function $a(t)$ from \eqref{i1} increases strictly monotone and satisfies $a(t)>a(0)=\frac12$ for any $t>0$. As a direct consequence of Theorem~\ref{theo1}, Corollary~1 in \cite{Kuznetsov2022} states that for $\mu>0$ any $\mu-$panharmonic function $u$ with $u\geq 0$ that does not identically vanish inside $\Omega$ satisfies the inequality
\begin{align}\label{ineq1}
\frac12 u(x) < \fint_{D_r(x)} u(y) \log\frac{r}{|x-y|}\dd y
\end{align}
for any admissible disc $D_r(x)\subset \Omega$. As a matter of fact, any nonnegative panharmonic function is subharmonic, that is, $-\Delta u \le 0$ (see e.g.~\cite[Theorem~1 and Remark 2]{Kuznetsov2022}). Kuznetsov therefore conjectured that inequality \eqref{ineq1} also holds for any subharmonic function $u\geq 0$ that does not vanish identically in $\Omega$. However, this is not true; in fact, Lemma~\ref{lem} directly forces the following
\begin{corollary}
Let $\Omega=D_1(0) \subset \R^2$ and set $u(x_1,x_2)=e^{x_1}\sin(x_2)+3$. Then $u$ is harmonic (in particular subharmonic) with $u\geq 0$ in $\Omega$, but \eqref{ineq1} does not hold for any $x\in \Omega$.
\end{corollary}
\begin{proof}
A short calculation shows that $u\geq 0$ in $\Omega$. Moreover, obviously, $\Delta u=0$, so the assumptions of Lemma~\ref{lem} are satisfied and we conclude easily.
\end{proof}
\begin{rem}
Since $e^{x_1}\sin(x_2)$ is analytic, we obviously can exchange $\Omega=D_1(0)$ in the previous Corollary by any domain $\Omega\subset\{(x_1,...,x_d)\in \R^d: x_1<c\}$ for some $c\in\R$ and recall the proof for $\tilde{u}(x_1,...,x_d)=e^{x_1}\sin(x_2) - \min_{(x_1,...,x_d)\in\Omega} e^{x_1}\sin(x_2)$.
\end{rem}
The fact that inequality \eqref{ineq1} holds for $\mu-$panharmonic functions with $\mu>0$ is due to $a(t)>a(0)=\frac12$ for any $t>0$ and the MVP \eqref{i1}. Of course, one might instead ask whether inequality \eqref{ineq1} holds for any function $u\geq 0$ with $u\not \equiv 0$ which is \emph{strictly} subharmonic, i.e., $-\Delta u \lneq 0$ in $\Omega$.

\section{Proof of Lemma \ref{lem}}
Before proving Lemma \ref{lem}, we recall the well-known fact that harmonic functions satisfy both the mean value property and the spherical mean value property, i.e.,
\begin{align}
u(x) &= \fint_{D_r(x)} u(y) \dd y = \frac{1}{\omega_d r^d} \int_{D_r(x)} u(y) \dd y \label{1}\\
&= \fint_{\del D_r(x)} u(y) \dd \sigma(y) = \frac{1}{d \omega_d r^{d-1}}\int_{\del D_r(x)} u(y) \dd \sigma(y), \label{tri}
\end{align}
for any admissible $r>0$. Note further that by rescaling, we have from \eqref{tri}
\begin{align}\label{2}
u(x)=\fint_{\del D_r(x)} u(y) \dd \sigma(y) = \fint_{\del D_1(0)} u(x+r\phi) \dd \sigma(\phi) = \frac{1}{d \omega_d}\int_{\del D_1(0)} u(x+r\phi) \dd \sigma(\phi).
\end{align}

We are now in the position to prove Lemma \ref{lem}.
\begin{proof}[Proof of Lemma \ref{lem}]
First, note that
\begin{align*}
\fint_{D_r(x)} u(y) \log\frac{r}{|x-y|} \dd y &= \fint_{D_r(x)} u(y) \log r \dd y - \fint_{D_r(x)} u(y)\log |x-y| \dd y\\
&= u(x)\log r - \fint_{D_r(x)} u(y) \log |x-y| \dd y
\end{align*}
since $u$ is harmonic and so satisfies \eqref{1}. In turn, it is enough to show
\begin{align*}
\fint_{D_r(x)} u(y) \log |x-y| \dd y = u(x) \bigg[\log r-\frac1d\bigg].
\end{align*}
For the sequel, we set without loss of generality $x=0$ (otherwise do a transformation $z=x-y$ in the integrals and repeat the computations for $v_x(z)=u(x-z)$). Using that $u$ also satisfies \eqref{2}, we deduce
\begin{align*}
\int_{D_r(0)} u(y)\log|y| \dd y &\overset{\phantom{\eqref{2}}}{=} \int_0^r \int_{\del D_1(0)} u(s \phi) \log s \cdot s^{d-1}\dd \sigma(\phi) \dd s = \int_0^r s^{d-1}\log s \int_{\del D_1(0)} u(s \phi) \dd \sigma(\phi) \dd s\\
&\overset{\eqref{2}}{=} d \omega_d u(0) \int_0^r s^{d-1}\log s \dd s = d \omega_d u(0)\bigg[\frac{r^d}{d}\log r-\frac{r^d}{d^2}\bigg] = \omega_d r^d u(0) \bigg[\log r-\frac1d\bigg].
\end{align*}
Dividing by $\omega_d r^d$, this finishes the proof.
\end{proof}

\section{Acknowledgements}
The author was supported by the Czech Science Foundation (GA\v{C}R) project GA22-01591S. The Institute of Mathematics, CAS is supported by RVO:67985840.

\bibliographystyle{amsalpha}

\end{document}